\documentclass{amsart}

\usepackage{amsmath}
\usepackage{amsthm}
\usepackage{amssymb}
\usepackage{amsbsy}
\usepackage{amsfonts}
\usepackage{amstext}
\usepackage{amscd}
\usepackage{tikz}
\usepackage{graphicx}

\numberwithin{equation}{section}
\theoremstyle{plain}
\newtheorem{thm}{Theorem}[section]
\newtheorem{prop}[thm]{Proposition}

\newtheorem{lem}[thm]{Lemma}
\theoremstyle{definition}

\newcommand{\real}{\mathbb{R}}

\begin{document}
\title[Distance $k$-graphs of star products]{Asymptotic spectral distributions of distance $k$-graphs of star product graphs.}

\author{Octavio Arizmendi}
\author{Tulio Gaxiola}

\address{Research Center for Mathematics, CIMAT, Aparatado Postal 402, Guanajuato, GTO, 36240, Mexico}

\date{\today}
\maketitle

\begin{abstract}
Let $G$ be a finite connected graph and let $G^{[\star N,k]}$ be the distance $k$-graph of the $N$-fold star power of $G$. For a fixed $k\geq1$, we show that the large $N$ limit of the spectral distribution of $G^{[\star N,k]}$ converges to a centered Bernoulli distribution, $1/2\delta_{-1}+1/2\delta_1$.  The proof is based in a fourth moment lemma for convergence to a centered Bernoulli distribution.
\end{abstract}

\section{Introduction and Statement of Results}

The interest in asymptotic aspects of growing combinatorial objects has increased in recent years.  In particular, the asymptotic spectral distribution of graphs has been studied from the quantum probabilistic point of view, see Hora \cite{HO} and Hora and Obata \cite{HO}.   Moreover, as observed in Accardi, A. Ben Ghorbal and Obata \cite{ABO}, Obata \cite{O} and Accardi, Lenczewski and Salapata \cite{ALS},   the cartesian, star, rooted and free products of graphs correspond to natural independences in non-commutative probability, see \cite{BGSh,M, S}. This has lead to state  central limit theorems for these product of graphs by reinterpreting the classical, free \cite{B, V}, Boolean \cite{SW} and monotone \cite{M2} central limit theorems.

More recently, in a series of papers \cite{Hi,HLO,K,KH,LO,O2}  the asymptotic spectral distribution of the distance $k$-graph of the $N$-fold power
of the cartesion product was studied.  These investigations, finally lead to the following theorem which generalizes the central limit theorem for cartesian products of graphs.

\begin{thm}[Hibino, Lee  and Obata \cite{HLO}] \label{HLO}
Let $G=(V,E)$ be a finite connected graph with $|V|\ge2$.
For $N\ge1$ and $k\ge1$ let $G^{[N,k]}$ be the distance $k$-graph of $G^N=G\times \cdots\times G$ 
($N$-fold Cartesian power)
and $A^{[N,k]}$ its adjacency matrix.
Then, for a fixed $k\ge1$, 
the eigenvalue distribution of $N^{-k/2}A^{[N,k]}$ converges in moments as $N\rightarrow\infty$
to the probability distribution of
\begin{equation}
\left(\frac{2|E|}{|V|}\right)^{k/2}
\frac{1}{k!}\Tilde{H}_k(g),
\end{equation}
where $\Tilde{H}_k$ is the monic Hermite polynomial of degree $k$ and $g$ is a random variable obeying the standard normal distribution $N(0,1)$.
\end{thm}

In this note we study the distribution (with respect to the vacuum state) of the star product of graphs. That is, we prove the analog of Theorem \ref{HLO} by changing the cartesian product by the star product.

\begin{thm}\label{T1} Let $G=(V,E,e)$ be a locally finite connected graph and let $k\in\mathbb{N}$ be such that  $G^{[k]}$ is not trivial.  For $N\ge1$ and $k\ge1$ let $G^{[\star N,k]}$ be the distance $k$-graph of $G^{\star N}=G\star \cdots\star G$  ($N$-fold star power)
and $A^{[\star N,k]}$ its adjacency matrix. Furthermore, let $\sigma=V_{e}^{\left[ k\right] }$ be the number of neighbours of $e$ in the distance $k$-graph of $G,$
then the distribution with respect to the vacuum state of $(N\sigma)^{-1/2}A^{[\star N,k]}$ converges in distribution as $N\rightarrow\infty$ to a centered Bernoulli distribution. That is,
\begin{equation*}\label{limit}
\frac{A^{\left[ \star N,k\right] }}{\sqrt{N\sigma}}
\longrightarrow \frac{1}{2} \delta _{-1}+\frac{1}{2}\delta _{1},
\end{equation*}
weakly.
\end{thm}

The limit distribution above  is universal in the sense that it is independent of the details of a factor $G$, but also in this case  the limit \textbf{does not} depend on $k$.  The proof of Theorem \ref{T1} is based in a fourth moment lemma for convergence to a centered Bernoulli distribution.

Apart from the introduction, this note is organized as follows. Section 2 is devoted to preliminaries. In Section 3 we prove a fourth moment lemma for convergence to a centered Bernoulli distribution. Finally in Section 4 we prove Theorem \ref{T1}.

\section{Preliminaries}

In this section we give very basic preliminaries on graphs the Cauchy transform, Jacobi parameters and non-commutative probability. The reader familiar with these objects may skip this section.

\subsection{Graphs}

A \emph{directed graph} or \emph{digraph} is a pair $G=\left( V,E\right) $ such that $V$ is a non-empty
set and $E\subseteq V\times V.$ The elements of $V$ and $E$ are called the 
\emph{vertices} and the \emph{edges }of the digraph $G,$ respectively.
Two vertices $x,$ $y\in V$ are \emph{adjacent, }or \emph{neighbours }if $%
\left( x,y\right) \in E.$

We call \emph{loop} an edge of the form $\left( v,v\right) $ and we say that
a graph is \emph{simple }if has no loops. A digraph is called \emph{undirected 
} if $\left( v,w\right) \in E$ implies $\left( w,v\right) \in E.$ 

We will work with simple undirected digraphs and use the word 
\emph{graph} for a simple undirected digraph without any further reference.

The \emph{adjacency matrix} of $G$ is the matrix indexed by the vertex set $V
$, where $A_{xy}=1$ when $\left( x,y\right) \in E$ and $A_{xy}=0$ otherwise.

A \emph{path }is a graph $P=\left( V,E\right) $ with vertex set $V=\left\{ v_{1},
\dots , v_{k}\right\} $ and edges $E=\{ ( v_{1},v_{2}),
\dots ,( v_{k-1},v_{k}) \} .$
A \emph{walk }is a path that can repeat edges. We say that a graph is \emph{%
connected } if every pair of distinct vertices $x,$\ $y\in V$ are connected
by a walk (or equivalently by a path).

In this note we focus on specific type of graphs coming from the distance $k$-graphs
of the star product of finite rooted graph. 

For a given graph $G=\left( V,E\right) $ and a positive integer $k$ the 
\emph{distance }$k$\emph{-graph }is defined to be a graph $G^{\left[ k\right]
}=\left( V,E^{\left[ k\right] }\right) $ with%
\[
E^{\left[ k\right] }=\left\{ \left( x,y\right) :x,\ y\in V,\ \partial
_{G}\left( x,y\right) =k\right\} ,
\]%
where $\partial _{G}\left( x,y\right) $ is the graph distance. Figure \ref{fig1} shows the distance $2$-graph induced by the 3 dimentional cube.

\begin{figure}[here]
\begin{center}
\includegraphics[height=3cm]{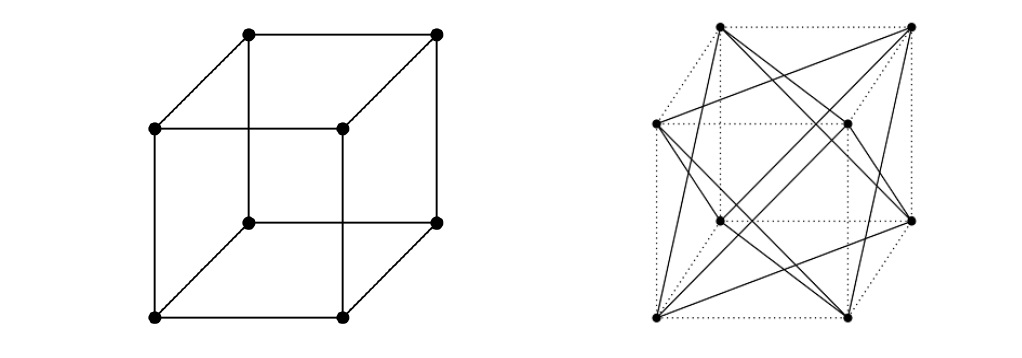}
\caption{\label{fig1}3-Cube and its distance $2$-graph}
\end{center}
\end{figure}

A  \emph{rooted} graph is a graph with a labeled vertex $o\in V$. 
Finally we define the star product of rooted graphs. For $G_{1}=\left( V_{1},E_{1}\right) $ and $%
G_{2}=\left( V_{2},E_{2}\right) $ be two graph with distinguished vertices $%
o_{1}\in V_{1}$ and $o_{2}\in V_{2},$ the \emph{star product graph }of $G_{1}
$ with $G_{2}$ is the graph $G_{1}\star G_{2}=\left( V_{1}\times
V_{2},E\right) $ such that for $\left( v_{1},w_{1}\right) ,\ \left(
v_{2},w_{2}\right) \in V_{1}\times V_{2}$ the edge $e=\left(
v_{1},w_{1}\right) \sim \left( v_{2},w_{2}\right) \in E$ if and only if one
of the following holds:%
\begin{eqnarray*}
1.\ v_{1} &=&v_{2}=o_{1}\ \text{and\ }w_{1}\sim w_{2} \\
2.\ v_{1} &\sim &v_{2}\ \text{and\ }w_{1}=w_{2}=o_{2}.
\end{eqnarray*}

As we can see, the star product is a graph obtained by gluing two graphs at
their distinguished vertices $o_{1}$ and $o_{2}.$

\begin{figure}[here]
\begin{center}
\includegraphics[height=3cm]{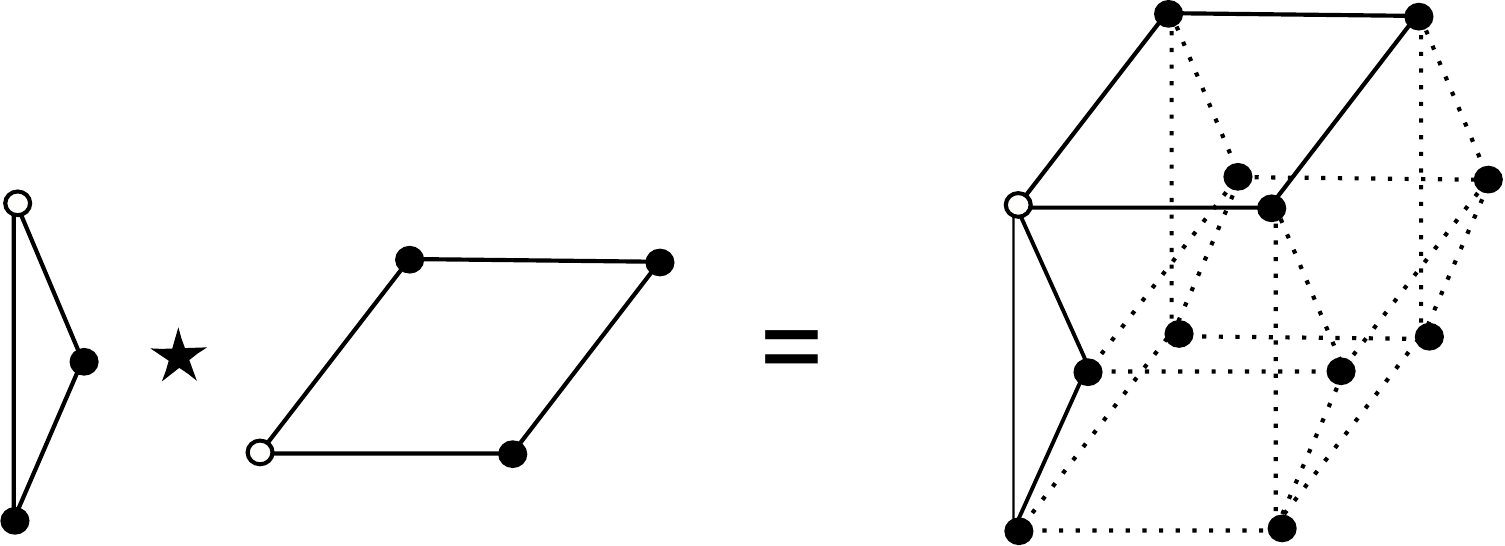}
\caption{\label{fig2}Star product of two cycles}
\end{center}
\end{figure}

\subsection{The Cauchy Transfom }

We denote by $\mathcal{M}$ the set of Borel probability measures on $\real$. The upper half-plane and the lower half-plane are respectively denoted as $\mathbb{C}^+$ and $\mathbb{C}^-$. 

For a measure $\mu \in \mathcal{M}$, the Cauchy transform $G_\mu:\mathbb{C}^+\to\mathbb{C}^-$  is defined by the integral 
$$G_\mu(z) = \int_{\real}\frac{\mu(dx)}{z-x}, ~~~~ z \in \mathbb{C}^+)$$

The Cauchy transform is an important tool in non-commutative probability. For us, the following relation between weak convergence and the Cauchy Transform will be important.

\begin{prop}
\label{WCCT}Let $\mu _{1}$ and $\mu _{2}$ be two probability measures on $%
\mathbb{R}$ and  
\begin{equation}\label{distance}
d(\mu _{1},\mu _{2})=\sup \left\{ \left\vert G_{\mu _{1}}(z)-G_{\mu
_{2}}(z)\right\vert ;\Im(z)\geq 1\right\} .
\end{equation}
Then $d$ is a distance which defines a metric for the weak topology of probability
measures. Moreover,  $|G_\mu(z)|$ is bounded in $\{z:\Im(z)\geq 1\}$ by $1$. \end{prop}
In other words, a sequence of probability measures $\left\{ \mu
_{n}\right\} _{n\geq 1}$ on $\mathbb{R}$ converges weakly to a probability
measure $\mu $ on $\mathbb{R}$ if and only if for all $z$ with $\Im(z)\geq 1$ we have
\begin{equation*}
\lim_{n\rightarrow \infty }G_{\mu _{n}}(z)=G_{\mu }(z).
\end{equation*}

\subsection{The Jacobi Parameters}

Let $\mu$ be a probability measure with all moments, that is $m_n(\mu):=\int_{\mathbb{R}}x^{n}\mu (dx)<\infty$. The Jacobi parameters  $\gamma _{m}=\gamma _{m}(\mu )\geq 0,\beta
_{m}=\beta _{m}(\mu )\in\mathbb{R}$, are defined by the recursion 
\begin{equation*}
xP_{m}(x)=P_{m+1}(x)+\beta _{m}P_{m}(x)+\gamma _{m-1}P_{m-1}(x),
\end{equation*}%
where the polynomials  $P_{-1}(x)=0,$ $P_{0}(x)=1$ y $(P_{m})_{m\geq 0}$ is a sequence of orthogonal monic polynomials with respect to $\mu $, that is,
\begin{equation*}
\int_{\mathbb{R}}P_{m}(x)P_{n}(x)\mu (dx)=0\text{ \  \  \  \ si }m\neq n.
\end{equation*}
A measure $\mu$ is supported on $m$ points iff $\gamma_{m-1} = 0$ and $\gamma_n> 0$ for $n = 0,\dots , m-2$.

The Cauchy transform may be expressed as a continued fraction in terms of the Jacobi parameters, as follows.
\begin{equation*}
G_{\mu }(z)=\int_{-\infty }^{\infty }\frac{1}{z-t}\mu (dt)=\frac{1}{z-\beta
_{0}-\dfrac{\gamma _{0}}{z-\beta _{1}-\dfrac{\gamma _{1}}{z-\beta _{2}-\cdots}}}
\end{equation*}

An important example for this paper is the Bernoulli distribution $\mathbf{b}=1/2\delta_1+1/2\delta_1$ for which
$ \beta_0=0, \gamma_0=1,$ and  $\beta_n=\gamma_n=0$ for $n\geq1$. Thus, the Cauchy transform is given by $$G_\mathbf{b}(z)=\frac{1}{z-1/z}.$$

In the case when $\mu$ has $2n+2$-moments we can still make an orthogonalization procedure until the level $n$. In this case the Cauchy transform has the form

\begin{equation}\label{expansionjacobi}
G_{\mu }(z)=\frac{1}{z-\beta
_{0}-\dfrac{\gamma _{0}}{z-\beta _{1}-\dfrac{\gamma _{1}}{~~~\dfrac{\ddots}{ z-\beta _{n}-
\gamma _{n}G_{\nu }(z)}}}}
\end{equation}
where $\nu$ is a probability measure.

\subsection{Non-Commutative Probability Spaces}

A $C^*$\textit{-probability space} is a pair $(\mathcal{A},\varphi)$, where $\mathcal{A}$ is a unital $C^*$-algebra and $\varphi:\mathcal{A}\to\mathbb{C}$ is a positive unital linear functional. The elements of $\mathcal{A}$ are called (non-commutative) random variables. An element $a\in\mathcal{A}$ such that $a=a^*$ is called self-adjoint.

The functional $\varphi$ should be understood as the expectation in classical probability.

 For $a_1,\dots,a_k\in \mathcal{A}$, we will refer to the values of $\varphi(a_{i_1}\cdots a_{i_n})$, $1\leq i_1,...i_{n}\leq k$, $n\geq1$, as the \textit{mixed moments} of $a_1,\dots,a_k$.

For any self-adjoint element $a\in\mathcal{A}$ there exists a unique probability measure $\mu_a$ (its spectral distribution) with the same moments as $a$, that is, $$\int_{\mathbb{R}}x^{k}\mu_a (dx)=\varphi (a^{k}), \quad \forall k\in \mathbb{N}.$$

We say that a sequence $a_n\in\mathcal{A}_n$ \emph{converges in distribution} to $a\in\mathcal{A}$ if $\mu_{a_n}$ converges in distribution to $\mu_a$. 

In this note we will only consider the $C^*$-probability spaces $(\mathcal{M}_n,\phi_1)$, where $\mathcal{M}_n$ is the set of matrices of size $n\times n$ and for a matrix $M\in\mathcal{M}_n$ the functional $\phi_{1}$ evaluated in $M$ is given by $$\phi_{1}(M)=M_{11}.$$

Let $G = (V,E,1)$ be a finite rooted graph with vertex set $\{1, ..., n\}$ and let $A_G$ be the adjacency matrix. We denote by $A(G)\subset \mathcal{M}_n$
be the adjacency algebra, i.e., the $*$-algebra generated by $A_G$. 

It is easy to see that the $k$-th moment of $A$ with respect to the $\phi_{1}$  is given the the number of walks in $G$ of size $k$ starting and ending at the vertex $1$. That is, 
$$\phi_{1}(A^k)=|\{(v_1,...,v_k): v_1=v_k=1 ~ and ~ (v_i,v_{i+1})\in E\}|.$$

Thus one can get combinatorial information of $G$ from the values of $\phi_1$ in elements of $A(G)$ and vice versa.

\section{The fourth moment lemma}

The following lemma which shows that the first, second and fourth moments are enough to ensure convergence to a Bernoulli distribution was observed in \cite{Ar1} . We give a new  proof  in terms of Jacobi parameters for the convenience of the reader.
\begin{lem}\label{lemma1}
 Let  $\left\{ X_{n}\right\} _{n\geq 1}\subset \left(\mathcal{A},\varphi \right) ,$  be a sequence of self-adjoint random variables in some non-commutative probability space, such that $\varphi(X_n)=0$  and $\varphi(X_n^2)=1$. If $\varphi\left( X_{n}^{4}\right)
\to1,$ as $n\to \infty$, then $\mu_{X_n}$ converges in distribution to a symmetric Bernoulli random variable $\mathbf{b}$.
\end{lem}

\begin{proof}
Let $\left( \left\{ \gamma _{i}\left( \mu _{X_{n}}\right) \right\} ,\left\{
\beta _{i}\left( \mu _{X_{n}}\right) \right\} \right) $ be the Jacobi parameters of the measures $\mu _{X_{n}}.$ The first moments $\{m_n\}_{n\geq1}$ are given in terms of the Jacobi Parameters as follows, see \cite{AcBo}.
\begin{eqnarray*}
m_{1} &=&\beta _{0} \\
m_{2} &=&\beta _{0}^{2}+\gamma _{0} \\
m_{3} &=&\beta _{0}^{3}+2\gamma _{0}\gamma _{0}+\beta _{1}\gamma _{0} \\
m_{4} &=&\beta _{0}^{4}+3\beta _{0}^{2}\beta _{1}+2\beta _{1}\beta
_{0}\gamma _{0}+\beta _{1}^{2}\gamma _{0}+\gamma _{0}^{2}+\gamma _{0}\gamma
_{1}.
\end{eqnarray*}%
Since $m_{1}\left( \mu _{X_{n}}\right) =0$ and $m_{2}\left( \mu
_{X_{n}}\right) =1$ we have
\begin{equation*}
\beta _{0}\left( \mu _{X_{n}}\right) =0\ \ \ \text{and\ \ }\ \gamma
_{0}\left( \mu _{X_{n}}\right) =1\ \ \ \ \ \forall n\geq 1,
\end{equation*}%
Hence, 
\begin{equation}\label{jpineq}
m_{4}\left( \mu _{X_{n}}\right) =\beta _{1}^{2}\left( \mu _{X_{n}}\right)
+1+\gamma _{1}\left( \mu _{X_{n}}\right) .
\end{equation}%
Now, since $m_{4}\left( \mu _{X_{n}}\right)1$ and $\gamma _{1}\geq 0$ we have the convergence
\begin{equation*}
\beta _{1}\left( \mu _{X_{n}}\right) \underset{n\rightarrow \infty }{%
\to}0\ \ \ \text{and\ \ \ }\gamma_{1}\left( \mu _{X_{n}}\right) 
\underset{n\rightarrow \infty }{\to }0.
\end{equation*}
Let $G_{\mu_n}$ be the Cauchy transform of $\mu_n $. By (\ref{expansionjacobi}) we can expand $G_\mu$ as a continued fraction as follows 
$$G_{\mu_n }(z)=\frac{1}{z-\dfrac{1}{ z-\beta _{1}-\gamma _{1}G_{\nu_n }(z)}} $$
where $\nu_n$ is some probability measure. Now, recall that $|G_{\nu_n }(z)|$ is bounded by $1$ in the set $\{z\vert ;\Im(z)\geq 1\}$ and thus, since $\gamma_1\rightarrow0$ and $\beta_1\to0$ we see that $\gamma _{n}G_{\nu_n }(z)\rightarrow 0 $. This implies the point-wise convergence
$$G_{\mu_n }(z)\to\frac{1}{z-\dfrac{1}{ z}}$$
 in the set $\{z\vert ;\Im(z)\geq 1\}$, which then implies the weakly convergence $\mu_n\to\mathbf{b}$.
\end{proof}

From the proof of the previous lemma one can give a quantitative version in terms of the distance given in eq (\ref{distance}).

\begin{prop}
Let $\mu$ be a probability measure such that $m_4:=m_4(\mu)$ is finite.  Then 

\begin{equation*}
d(\mu,\frac{1}{2}\delta_1+\frac{1}{2}\delta_{-1})\leq 4\sqrt{ m_4-1}.
\end{equation*}
\end{prop}

\begin{proof}

If $m_4-1>1/16$  then the statement is trivial since $d(\mu,1/2\delta_1+1/2\delta_{-1})\leq1$ for any  measure $\mu$. Thus we may assume that $(m_4-1)\leq1/16$.

Denoting by $f(z)=\beta _{1}-\gamma _{1}G_{\nu_n }(z)$  we have
$$|G_\mu(z)-G_b(z)|=
\left| \frac{1}{z-\dfrac{1}{ z}}- \frac{1}{z-\dfrac{1}{ z-f(z)}}\right|=\left|\frac{f(z)}{(z^2-1)(z^2-1-f(z)z)}\right|.$$

From (\ref{jpineq}) we get the inequalities  $\sqrt{m_4-1}\geq|\beta _{1}|$ and $\sqrt{m_4-1}\geq m_4-1\geq \gamma_1$. Since, for $Im(z)>1$, we have that, $|G_{\nu }(z)|<1$  we see that $|f(z)|=|\beta _{1}-\gamma _{1}G_{\nu}(z)|\leq 2\sqrt{m_4-1}\leq1/2$., from where we can easily obtain the bound $ |\frac{1}{z^2-1-f(z)z}|\leq2$.
Also, for $Im(z)>0$ we have the  bound $|\frac{1}{(z^2-1)}|<1$. Thus we have 

\begin{eqnarray}
|G_\mu(z)-G_b(z)|&=&\left|\frac{f(z)}{(z^2-1)(z^2-1-f(z)z)}\right|
\\
&=&\left|f(z)\right|\left|\frac{1}{(z^2-1)} \right| \left|\frac{1}{z^2-1-f(z)z}\right|
\\
&\leq& 2 |f(z)|\leq4\sqrt{m_4-1}.
\end{eqnarray}
as desired.

\end{proof}

\section{Proof of Theorem \ref{T1}.}

 Before proving Theorem \ref{T1} we will prove a lemma about the structure of distance $k$-graph of the iterated  star product of a graph.

\begin{lem}\label{descomp}
Let $G=(V,E,e)$  be a connected finite graph with root $e$ and $k$ such that $G^{[k]}$ is  a non-trivial graph.
 Let ${G^{\star N} }^{[k]}$ be the distance $k$-graph of the N-th star product of G,  then 
${G^{\star N} }^{[k]}$ admits a decomposition of the form.

$${G^{\star N} }^{[k]}=({G}^{[k]})^{\star N}\cup \hat G.$$

where $\partial G (z, e) < k$, for all $z\in \hat G$:

\end{lem}

\begin{proof}
Let $G_{1},\ G_{2},\ \dots ,$ $G_{N}$ be the $N$ copies of $G,$ that form
the star product graph $G^{\star N}$ by gluing them at $e.$ For $x,
y\in G_{i},$ the distance between $x$ and $y$ is given by
\[
\partial _{G^{\star N}}\left( x,y\right) =\partial _{G_{i}}\left( x,y\right)
=\partial _{G}\left( x,y\right) ,
\]
hence
\[
\left( x,y\right) \in E\left( G_i^{\left[ k\right] }\right) \ \text{if and
only if\ }\left( x,y\right) \in E\left( \left( G^{\star N}\right) ^{\left[ k%
\right] }\right) ,
\]
therefore we have $\left( G^{\left[ k\right] }\right) ^{\star N}\subseteq
\left( G^{\star N}\right) ^{\left[ k\right] }.$

Now, if $x\in G_{i}$ and $y\in G_{j}$ with $j\neq i,$ by definition all the
paths in $G^{\star N}$ from $x$ to $y$ must pass throw $e,$ then we have
\[
\partial _{G^{\star N}}\left( x,y\right) =\partial _{G_{i}}\left( x,e\right)
+\partial _{G_{j}}\left( y,e\right) ,
\]
thus
\[
\left( x,y\right) \in E\left( \left( G^{\star N}\right) ^{\left[ k\right]
}\right) \ \text{if and only if }\partial _{G_{i}}\left( x,e\right)
+\partial _{G_{j}}\left( y,e\right) =k.
\]
Since $\partial _{G_{i}}\left( x,e\right) ,\ \partial _{G_{j}}\left(
y,e\right) >0,$ we obtain the desired result.
\end{proof}

Now, we are in position to prove the main theorem of the paper. 

\begin{proof}[Proof  of Theorem \ref{T1}]

 Consider the non-commutative probability space $\left( \mathcal{A}
,\phi _{1}\right) $ with $\phi _{1}\left( M\right) =M_{11},$ for $M\in 
\mathcal{A}$ (see Section 2). Then, recall that,  if $ A $ is an adjacency matrix, $
\phi _{1}\left( A^{k}\right) $ equals the number of walks of  size 
$k$ starting and ending at the vertex  $1$.

Since $G$ is a simple graph, it has no loops and then  $G^{\star N}$ is also a simple graph. Thus, 
\begin{equation*}
\phi _{1}\left( \frac{A^{\left[ \star N,k\right] }}{\sqrt{N|V_{e}^{\left[
k\right] }|}}\right) =0.
\end{equation*}
Now, observe that since the graph $G^{\star N}$ has no loops,  the only walks in $G$ of size $2$ which start in $e$ and end in $e$
are of the form  $\left( exe\right)$, where $x$ is a neighbor of $e$ in  $\left(
G^{\star N}\right) ^{\left[ k\right] }$ .  The number of neighbors of $e$ is exactly $N|V_{e}^{\left[ k\right] }|$, thus 

\begin{eqnarray*}
\phi _{1}\left( \left( \frac{A^{\left[ \star N,k\right] }}{\sqrt{N|V_{e}^{
\left[ k\right] }|}}\right) ^{2}\right) &=&\frac{1}{N|V_{e}^{\left[ k\right]
}|}\phi _{1}\left( \left( A^{\left[ \star N,k\right] }\right) ^{2}\right)
\\
&=&\frac{1}{N|V_{e}^{\left[ k\right] }|}N|V_{e}^{\left[ k\right] }|=1.
\end{eqnarray*}

Thus we have seen that $\phi(A_N)=0$  and  $\phi(A_N^2)=1$.  Hence, it remains to  show that  $\phi(A_N^4)\to1$ as $N\to \infty$.

We are interested in counting the number of walk of size $4$ that start and finish at $e$ in $\left( G^{\star N}\right) ^{\left[ k\right] }.$  We will divide this walks in two types.

\begin{figure}[here]
\begin{center}
\includegraphics[height=3cm]{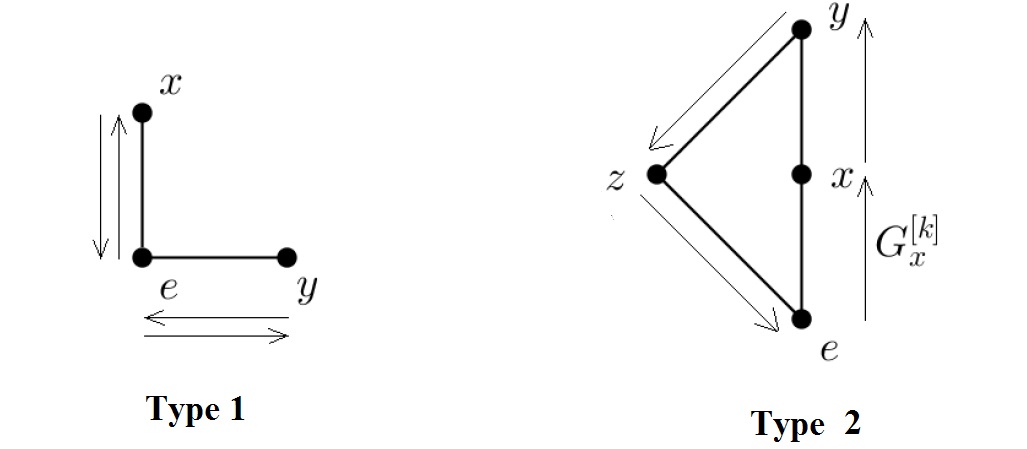}
\caption{\label{fig3} Types of walks of size 4}
\end{center}
\end{figure}

Type 1.  The first type of walk is of the form $exeye$. That is, the walk starts at $e$, then visits a neighbor $x$ of $e$ to then come back to $e$, this can be done in $N|V_{e}^{\left[ k\right] }|$ ways. After this, he again visits a a neighbour $y$ (which could be again $x$) of $e$ to finally come back to $e$. Again, this second step can be done in $N|V_{e}^{\left[ k\right] }|$ different ways , so there is $\left(
N|V_{e}^{\left[ k\right] }|\right) ^{2}$ walks of this type. Thus 
\begin{equation*}
1=\frac{\left( N|V_{e}^{\left[ k\right] }|\right) ^{2}}{\left( N|V_{e}^{%
\left[ k\right] }|\right) ^{2}}\leq \phi _{1}\left( \left( \frac{A^{\left[
\star N,k\right] }}{\sqrt{N|V_{e}^{\left[ k\right] }|}}\right) ^{4}\right) .
\end{equation*}

Type 2. Let $G_{x}^{\left[ k\right] }$ be the copy of  $G^{\left[ k\right] }$ in the distance-$k$  graph of the 
star product $\left( G^{\left[ k\right] }\right) ^{\star N}$ which contains $x$. The second type of walks is as follows. From $e$ is goes to some 
$x\in V_{e}^{\left[ k\right] }\ $ (which can be chosen in  $N|V_{e}^{\left[k\right] }|$ different ways),  and then from $x$ then he goes to some $y \neq e$. This $y$ should belong to $G_{x}^{
\left[ k\right] }\ $. Indeed, since $\delta _{\left( G^{\star N}\right)
}\left( e,x\right) =k,$ if $y$ would be in another copy of $G^{\left[
k\right] }$ the distance $\delta _{\left( G^{\star N}\right) }\left(
y,x\right) ,$ between $y$ and $x$ would be bigger than $k.$  The number of ways of choosing $y$ is bounded by the number of neighbours of $x$.

For the next step of the walk, from $y$ we can only go to a neighbor of $e$, say  $z\in V_{e}^{\left[ k\right] }$\ (since in the last step it must come back to $e$).  This $z$ indeed must also belong to $G_{x}^{\left[ k\right] },$. If this wouldn't be the case and $z\notin G_{x}^{\left[ k\right] }$, then we would have that $%
\delta _{\left( G^{\star N}\right) }\left( e,z\right) \neq k,$  which is a contradiction because of Lemma  \ref{descomp}.

\begin{figure}[here]
\begin{center}
\includegraphics[height=3.5cm]{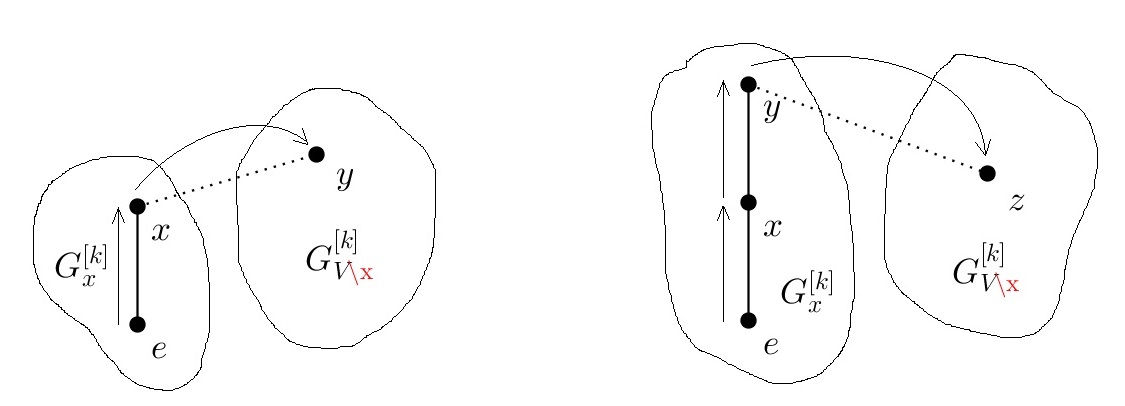}
\caption{\label{fig3} Obstructions}
\end{center}
\end{figure}

Finally, let $M=\max_{x\in V}|V_{x}^{
\left[ k\right] }|$. Then, from the above considerations we see that the number of  walks of Type 2  is bounded by  $ M \left( N|V_{e}^{\left[ k\right] }\right)  \left(
|V_{e}^{\left[ k\right] }|\right) ,$ from where 
\begin{eqnarray*}
\phi _{1}\left( \left( \frac{A^{\left[ \star N,k\right] }}{\sqrt{N|V_{e}^{%
\left[ k\right] }|}}\right) ^{4}\right) &\leq &\frac{\left( N|V_{e}^{\left[ k%
\right] }|\right) ^{2}}{\left( N|V_{e}^{\left[ k\right] }|\right) ^{2}}+%
\frac{N|V_{e}^{\left[ k\right] }|M|V_{e}^{\left[ k\right] }|}{\left(
N|V_{e}^{\left[ k\right] }|\right) ^{2}} \\
&=&1+\frac{M}{N}\underset{N\rightarrow \infty }{\longrightarrow }1,
\end{eqnarray*}%
since  $M$ does not depend on $N.$ Thanks to  Lemma \ref{lemma1} we obtain the desired result.
\end{proof}

\end{document}